\documentclass[11pt]{amsart}
\usepackage[shortalphabetic]{amsrefs}
\usepackage[utf8]{inputenc}
\usepackage{amsthm}
\usepackage[utf8]{inputenc}
\usepackage{appendix}
\usepackage[top= 2.5cm, bottom=2.5cm, left=2.5cm, right=2.5cm]{geometry}
\usepackage{cite}
\usepackage{amssymb,amsmath,amsthm}
\usepackage{hyperref}
\newcommand{\amsprimary}[1]{{\footnotesize\noindent AMS 2020 \textit{Mathematics subject
classification:} Primary #1\vspace{1pc}}}
\newcommand{\keywordsnames}[1]{{\footnotesize\noindent\textit{Key words:} #1\vspace{1pc}}}

\newtheorem{theorem}{Theorem}

\newtheorem{teo}{Theorem}

\newtheorem{corollary}[teo]{Corollary}
\newtheorem{lemma}[teo]{Lemma}

\theoremstyle{definition}

\title[]{Liouville theorem for biharmonic functions on manifolds of nonnegative Ricci curvature}
\author{John E. Bravo and Jean C. Cortissoz }
\email{j.bravob@uniandes.edu.co, jcortiss@uniandes.edu.co}
\address{Department of Mathematics, Universidad de los Andes, Bogot\'a DC, Colombia}
\date{}

\begin{document}

%In this paper we explore Liouville's theorem on Riemannian cones as defined below.
%We also study the Strong Liouville Property, that is, the property of a
% cone having spaces of harmonic functions of a fixed
%polynomial growth of finite dimension.
%We explore Liouville’s theorem and the Strong Liouville Property (SLP) for harmonic functions on Riemannian cones and surfaces. Our approach reframes the classical Liouville property in terms of the growth of radial eigenfunctions 
%(in the case of manifolds with
%rotational symmetry), allowing us to recover and sharpen known results under minimal assumptions. We provide explicit estimates for the slowest-growing non-constant harmonic functions on cones and surfaces, and construct examples where doubling fails but Liouville and SLP still hold. Finally, we extend the analysis to 
%$p$-subharmonic functions, proving a nonlinear Liouville theorem %under curvature conditions resembling Milnor's celebrated criterion
%for harmonic functions.
\begin{abstract}
In this paper we extend Yau's celebrated Liouville theorem to the biharmonic case. Namely, we show that
in a complete Riemannian manifold with a pole and nonnegative Ricci curvature, any biharmonic function
of subquadratic growth must be harmonic, and hence, any biharmonic function of sublinear growth must be constant.
Our proof relies on a new local $L^2$ estimate for the Laplacian of biharmonic functions combined with a mean value inequality. Examples
where our theorem applies include hypersurfaces 
of positive sectional curvature in $\mathbb{R}^n$, and manifolds with a pole of nonnegative Ricci curvature
whose curvature decays at infinity rapidly enough.
\end{abstract}

\maketitle

\keywordsnames{Liouville theorem; bounded biharmonic functions; subquadratic growth.}

{\amsprimary {31C05, 53C21, 35B53}}

\tableofcontents

\section{Introduction}
%=========================================================
%  Caccioppoli inequality for the Laplacian of a biharmonic
%  function (local version with dependence on R-r and Ricci bound)
%=========================================================

% ======= Minimal preamble suggestions =======
% \usepackage{amsmath, amssymb, amsthm}
% \numberwithin{equation}{section}
% \newtheorem{prop}{Proposition}[section]
% \newtheorem{rem}[prop]{Remark}
% ============================================
%\section{Introduction}

The Liouville theorem, a landmark result of classical analysis, asserts that any bounded harmonic function on $\mathbb{R}^n$ is constant. In 1975, Yau \cite{Yau75} proved that this theorem continues to hold for complete Riemannian manifolds with nonnegative Ricci curvature, that is, every bounded harmonic function must be constant. Since then, the study of Liouville-type properties has been a central theme in geometric analysis, with deep connections to potential theory, volume growth, and curvature conditions (see Li–Schoen \cite{LiSchoen84}, Li–Tam \cite{LiTam1989}, Colding-Minicozzi 
\cite{ColdingMinicozzi97}, among many others).

\medskip
The case of higher-order elliptic equations, however, is considerably less developed. In Euclidean space it has long been known that bounded polyharmonic functions must be constant (see Aronszajn–Creese–Lipkin \cite{ACL83}). In fact, explicit representations show that entire biharmonic functions with subquadratic growth reduce to harmonic polynomials of degree at most one, and hence are harmonic. These Euclidean results suggest that the natural growth threshold for biharmonic functions is quadratic, in analogy with the linear growth threshold for harmonic functions.

\medskip
On Riemannian manifolds, of course, the situation is subtler. While the study of polyharmonic operators arises naturally in conformal geometry, most notably through the Paneitz operator,
%and the GJMS hierarchy \cite{Paneitz,Branson} — 
general Liouville theorems for biharmonic functions under Ricci curvature conditions appear to be largely absent from the literature.
As notable exceptions we include the works Branding \cite{Branding18} where Liouville-type theorems 
for biharmonic maps between manifolds are proved under smallness conditions on different
Sobolev norms.
%Existing results focus either on conformally covariant settings, or on biharmonic maps under smallness assumptions on Sobolev norms. 
In addition, Maz’ya and Mayboroda \cite{MazyaMayboroda06} obtained fundamental $L^2$ estimates for solutions of biharmonic Poisson problems in Euclidean domains, which inspired our search for analogous inequalities in the Riemannian setting. Also, previous work by the authors \cite{BravoCortissoz}
extends Milnor's criterion for parabolicity of surfaces to biharmonic functions, and a curvature regime where any bounded biharmonic function must be harmonic is identified. We believe that the study 
of Liouville-type properties and its relation with geometric invariants for
higher order elliptic operators is a fundamental question in Geometric Analysis.

\medskip
The goal of this paper is to provide such a result. We extend Yau’s theorem to biharmonic functions of subquadratic growth on certain complete Riemannian manifolds with nonnegative Ricci curvature. By subquadratic growth we mean that
a function $u:M\longrightarrow \mathbb{R}$ satisfies
\[
\lim_{r\rightarrow\infty}\frac{\sup_{x\in B_r\left(p\right)}\left|u\left(x\right)\right|}{r^2}=0,
\]
where $B_r\left(p\right)$ is the ball of radius $r>0$ centered at $p$.

\medskip
To be more precise, we shall consider manifolds with a pole, that is
manifolds $M$ for which there is a point $p\in M$ such that
the exponential map
\[
\mbox{exp}_p:\, T_pM \longrightarrow M
\]
is a diffeomorphism. 
Then we have the following.

\begin{theorem}
\label{thm:main}
Let $(M,g)$ be a complete Riemannian manifold with 
$\mbox{Ric}\geq 0$. Then any biharmonic function $u \in C^4(M)$ of subquadratic growth must be harmonic. In particular, any biharmonic function of sublinear growth must be constant.
\end{theorem}

\medskip
Our approach combines three key ingredients: (i) a local energy inequality for biharmonic functions, (ii) a hole-filling iteration lemma with singular factors, and (iii) a mean value inequality applied to the Laplacian of $u$. A central step is an $L^2$-estimate for $\Delta u$, which may be viewed as the biharmonic analogue of the classical Caccioppoli inequality
(see \eqref{ineq:Cacciopoli_laplacian}). This estimate is reminiscent of Maz’ya–Mayboroda’s work in the Euclidean case, but adapted here to the setting of manifolds with Ricci curvature bounded below. Once this estimate is in place, we show that subquadratic growth forces $\Delta u \equiv 0$, thereby reducing the problem to Yau’s theorem. We must observe that the proof of our Caccioppoli-type inequality \eqref{ineq:Cacciopoli_laplacian} is somewhat delicate, as the proof of Lemma \ref{lem:Caccio-1-biharmonic-Br} depends on a cut-off dependent
application of Young's inequality.

\medskip
In some sense Theorem \ref{thm:main} is sharp: in $\mathbb{R}^n$ any biharmonic function of subquadratic growth is in fact harmonic of linear growth, showing that the quadratic threshold cannot be weakened. Our theorem thus provides a natural extension of Liouville property to biharmonic functions in manifolds of nonnegative Ricci curvature, which is a natural higher-order extension of Yau’s Liouville theorem, and a foundational step toward $k$-polyharmonic analogues.

\medskip
\noindent\textbf{Organization of the paper.} The material in this paper is presented as follows. In Section \ref{section:technical} we collect some technical facts (which can be safely said that are 
 standard tools by now), in Section \ref{sect:Cacciopoli} we introduce
 some local energy estimates and Caccioppoli-type inequalities for
 biharmonic functions, and in Section \ref{sect:proof_main_thm} we prove our main result.

%%%%%%%%%%%%%%%%%%%%%%%%%%%%%%55

\section{Proof of the Main Result}

\subsection{Some Technical Lemmas}
\label{section:technical}

In this section, we collect some technical lemmas that will be used in the proof of our main result.

\medskip
We start with a classical cut-off construction, which is 
the key in the proof of our main theorem. But first, we need to make
an observation. If $(M,g)$ has a pole and $\mbox{Ric}_g\geq 0$ then $\Delta d\geq 0$,
where $r$ is the distance function from the pole. The reason is that,
along a geodesic $\gamma:\left[0,\infty\right)\longrightarrow M$ starting at the pole, the Laplacian of the
distance function from the pole $S(t)=(\Delta d)(\gamma(t))$ satisfies the Ricatti inequality
\[
S'\leq -\frac{S^2}{n-1}-\mbox{Ric}_g(\dot{\gamma},\dot{\gamma}),
\]
and hence 
\[
S'\leq -\frac{S^2}{n-1},
\]
so that if $S$ becomes negative at some point, then there would be a $t_0>0$ such that
\[
\lim_{t\rightarrow t_0^{+}}S(t)=-\infty,
\]
contradicting that the geodesic starts at a pole.

\begin{lemma}[Cutoff with explicit bounds]\label{lem:cutoff}
Let $(M^n,g)$ be a complete Riemannian manifold with 
a pole $p\in M$. Let $r$ be such that $0<r<R$.
Set $B_\rho:=B_\rho(p)$.
Assume $\mbox{Ric}_g\ge -(n-1)K\,g$ on $B_R$ for some $K\ge 0$.
Then there exists $\chi\in C_c^\infty(B_R)$ such that
\[
0\le \chi\le 1,\qquad
\chi\equiv 1\ \text{on } B_r,\qquad
\mbox{supp}\left(\chi\right)\subset B_R,
\]
and the derivative bounds, if $K>0$
\begin{equation*}%\label{eq:cutoff-bounds}
|\nabla \chi|\ \le\ \frac{C(n)}{R-r}, \quad \mbox{and}
\end{equation*}
\begin{equation*}
|\Delta \chi|\ \le\ C(n)\!\left(\frac{1}{(R-r)^2}+\frac{\max\left\{
\sqrt{K}\coth \sqrt{K},1\right\}}{R-r}\right)
\quad\text{on}\quad B_R ,
\end{equation*}
and if $K=0$
then 
\[
|\Delta\chi|\le C(n)\left((R-r)^{-2}+\frac{1}{(R-r)r}\right)
\quad \text{on} \quad B_R.
\]
Moreover, if $r\ge\frac{1}{8}R$, then
\[
|\Delta\chi|\le \frac{C(n)}{(R-r)^{2}}
\quad \text{on}\quad B_R.
\]
\end{lemma}

\begin{proof}
Let $p$ be a pole of $M$. Let $d(x):=d(x,p)$ be the distance function
to $p$. Fix a smooth cut off function
$\eta\in C^\infty(\mathbb{R})$ such that $0\le \eta\le 1$,
\[
\eta(t)=1\ \,\,\text{for }\,\, t\le 0,\qquad
\eta(t)=0\ \,\,\text{for }\,\, t\ge 1,\qquad
|\eta'|\le C_0,\ \ |\eta''|\le C_0 .
\]
Define the cutoff
\[
\chi(x):=\eta\!\left(\frac{d(x)-r}{R-r}\right).
\]
Then $0\le\chi\le 1$, $\chi\equiv 1$ on $B_r$, and $\mbox{supp}\left(\chi\right)\subset B_R$.
Since $p$ is a pole 
of $M$, $d$ is smooth away from $p$ and
$\nabla d$ exists with $|\nabla d|=1$. Hence,
$\chi$ is smooth and by the chain rule
\[
|\nabla\chi|\le \frac{C_0}{R-r}.
\]

Now, we focus on $\Delta \chi$.
We compute 
\[
\Delta \chi = \eta''\left(\frac{d-r}{R-r}\right)
\left<\nabla d, \nabla d\right>+\eta'\left(\frac{d-r}{R-r}\right)
\Delta d.
\]
By the Laplacian comparison under the curvature
inequality $\mbox{Ric}\ge -(n-1)K$,
\(
\Delta d \le (n-1)f_K(d),
\)
with the convention that 
$f_K(d)=\sqrt{K}\coth(\sqrt{K}\,d)$ if $K>0$ and
$f_K(d)=1/d$ if $K=0$. Thus,
\[
0\leq \Delta d \leq (n-1)f_K(d),
\]
and hence,
\[
|\Delta \chi|
\ \le\ \frac{C_0}{(R-r)^2}
+ \frac{C_0}{R-r}\,|\Delta d|
\ \le\ C(n)\!\left(\frac{1}{(R-r)^2}+\frac{\max\left\{f_K(d),1/d\right\}}{R-r}\right).
\]
If $K>0$, then 
\[
|\Delta \chi|
\ \le\ \ C(n)\!\left(\frac{1}{(R-r)^2}+\frac{\max\left\{
\sqrt{K}\coth \sqrt{K},1\right\}}{R-r}\right),
\]
that is
\[
|\Delta \chi|
\ \le\ C(n,K)\!\left(\frac{1}{(R-r)^2}+\frac{1}{R-r}\right),
\]
and if $K=0$, since only in the annulus 
$\eta'\neq 0$,
\[
|\Delta \chi|
\ \le\ C(n)\!\left(\frac{1}{(R-r)^2}+\frac{1}{(R-r)r}\right).
\]
Finally, if $r\ge\frac{1}{8}R$ then $c(R-r)< r$, we obtain that 
\[
|\Delta \chi|
\ \le \frac{C(n)}{(R-r)^2}.
\]
\end{proof}

From the previous lemma we obtain the existence of the following special type of
cutoff function.
\begin{corollary}
\label{corollary:special_cutoff}
Let $(M^n,g)$ be a complete Riemannian manifold with 
a pole $p\in M$.  Set $B_\rho:=B_\rho(p)$, assume $\mbox{Ric}_g\ge 0$ on $B_R$,
 and let $0<r<R$.
Then there exists $\varphi\in C_c^\infty(B_r)$ such that $0\leq \varphi\leq 1$,
$\varphi=1$ on $B_s$, and for $\dfrac{1}{8}r<s<r<R$,
    \[
    \left|\nabla \varphi\right|^2\leq \frac{C}{\left(r-s\right)^2}\varphi\quad
    \text{ and } \quad
    \frac{\left|\Delta\left(\varphi^4\right)\right|^2}{\varphi^4}\leq \frac{C}{\left(r-s\right)^4}.
    \]
\end{corollary}
\begin{proof}
    Take $\varphi=\chi^4$, with $\chi$
    as in Lemma \ref{lem:cutoff} . First we check: ha
    \begin{eqnarray*}
        \left|\nabla \varphi\right|^2&=& 16\left|\nabla \chi\right|^2 \chi^6\\
        &\leq & \frac{C}{(r-s)^2}\chi^4 = \frac{C}{(r-s)^2}\varphi.
    \end{eqnarray*}
    On the other hand,
    \begin{eqnarray*}
        \left|\Delta\left(\varphi^4\right)\right|^2&=& \left(4\varphi^3\Delta \varphi+12\phi^2\left|\nabla \varphi\right|\right)^2.
    \end{eqnarray*}
    %Hence,
    %\[
    %\chi^4\Delta \chi^4 = \chi^4\left(4\chi^3\Delta\chi+12\chi^2\left|\nabla \chi\right|^2\right)=4\chi^7\Delta\chi+12\chi^6%\left|\nabla \chi\right|^2.
    %\]
    Since $0\leq \chi\leq 1$ and from
    the bounds on $\nabla \chi$ and $\Delta \chi$, we can estimate
    \[
    \left|\Delta \chi^4\right|=\left|4\chi^3\Delta\chi+12\chi^2\left|\nabla \chi\right|^2\right|\leq 
    \frac{C}{\left(r-s\right)^2}.
    \]
    Then,
    \[
    4\left|\varphi^3\Delta \varphi\right|\leq \frac{C}{\left(r-s\right)^2} \varphi^2.
    \]
    Because
    \[
    \Delta \left(\varphi^4\right)=4\varphi^3\Delta\varphi+12\varphi^2\left|\nabla \varphi\right|^2,
    \]
    using all the previous estimates we obtain
    \[
    \left|\Delta\left(\varphi^4\right)\right|^2 \leq \left(\frac{C}{\left(r-s\right)^2} \varphi^2\right)^2
    \leq \frac{C}{\left(r-s\right)^4}\varphi^4,
    \]
    which finishes the proof of the corollary.
\end{proof}

\begin{lemma}[Hole-filling iteration]
\label{lemma:hole-filling}
Let $F,G:[0,R]\to[0,\infty)$ be bounded functions, $G$ nondecreasing, $\alpha >0$, and assume that there exists 
$\theta\in[0,1)$ such that for all $0<s<r\leq R$ one has
\begin{equation}\label{ineq:hole-filling}
    F(s) \;\leq\; \theta\,F(r) \;+\; \frac{1}{(r-s)^{\alpha}}\,G(r).
\end{equation}
%Suppose further that $G$ is bounded above by $M$. 
Then for any $\lambda\in(\theta^{1/\alpha},1)$,
\begin{equation}\label{ineq:iteration}
    F(s) \;\leq\; 
    \frac{(1-\lambda)^{-\alpha}}{1-\theta\lambda^{-\alpha}}\,
    \frac{G\left(r\right)}{(r-s)^{\alpha}}, 
    \qquad 0<s<r\leq R.
\end{equation}
\end{lemma}

\begin{proof}
Fix $0<s<r\leq R$ and define the geometric sequence 
\[
r_k := s + (1-\lambda^k)(r-s), \qquad k=0,1,2,\dots.
\]
Then $r_0=s$ and $r_k\uparrow r$. 
Applying \eqref{ineq:hole-filling} with $(s,r)=(r_k,r_{k+1})$ gives
\[
F(r_k)\;\leq\;\theta\,F(r_{k+1})
+ \frac{1}{(r_{k+1}-r_k)^{\alpha}}\,G(r_{k+1}).
\]
Since $r_{k+1}-r_k=(1-\lambda)\lambda^k(r-s)$, this becomes
\[
F(r_k)\;\leq\;\theta\,F(r_{k+1})
+ (1-\lambda)^{-\alpha}\lambda^{-\alpha k}(r-s)^{-\alpha}\,G(r_{k+1}).
\]
Iterating from $k=0$ to $k=m-1$ yields
\[
F(s)=F(r_0)\;\leq\;\theta^mF(r_m)
+(1-\lambda)^{-\alpha}(r-s)^{-\alpha}
\sum_{j=0}^{m-1} (\theta\lambda^{-\alpha})^j G(r_{j+1}).
\]
Using that $G$ is nondecreasing gives that $G(r_{j+1})\leq G(r)$ for all $j$. 
Hence
\[
F(s)\;\leq\;\theta^mF(r_m)
+(1-\lambda)^{-\alpha}(r-s)^{-\alpha}G(r)\sum_{j=0}^{m-1}(\theta\lambda^{-\alpha})^j.
\]
Since $\theta\lambda^{-\alpha}<1$, the series converges as $m\to\infty$ and 
$\theta^mF(r_m)\to0$. This proves \eqref{ineq:iteration}.
\end{proof}

The following basic estimate is Theorem 6.2
in Chapter II from \cite{SchoenYau1994}.

\begin{lemma}[Mean value inequality for $u^2$]
\label{lemma:mean_value_ineq}
Let $(M^n,g)$ be a complete Riemannian manifold and fix $x_0\in M$.
Assume a lower Ricci bound $\mathrm{Ric}_g \ge -(n-1)K g$ on $B_{2a}:=B_{2a}(x_0)$ for some $K\ge0$.
If $u\ge0$ is (weakly) subharmonic on $B_{2a}$ (i.e.\ $\Delta u\ge0$ in $B_{2a}$), then
\begin{equation}\label{eq:MV-u2}
\sup_{B_a} u^{2}
\;\le\;
C(n)\,\exp\!\big(C(n)\sqrt{K}\,R\big)\,
\frac{1}{\mathrm{Vol}(B_{2a})}\int_{B_{2a}} u^{2}.
\end{equation}
In particular, when $K=0$ (nonnegative Ricci curvature on $B_{2a}$),
\begin{equation}\label{eq:MV-u2-Ric0}
\sup_{B_a} u^{2}
\;\le\;
\frac{C(n)}{\mathrm{Vol}(B_{2a})}\int_{B_{2a}} u^{2}.
\end{equation}
\end{lemma}

\subsection{Local energy estimates and Cacciopoli-type inequalities}
\label{sect:Cacciopoli}
Next, we state and prove a local energy inequality for the Laplacian of a biharmonic function.
\begin{lemma}
\label{lem:laplacian_estimate}
Let $(M,g)$ be an $n$–dimensional complete Riemannian manifold with a pole $p$ with $\mathrm{Ric}\ge 0$. Let $B_{\rho}\colon = B_{\rho}\left(p\right)$.
Fix $a>0$ and let $u\in C^\infty(B_{2a})$ be biharmonic, i.e.\ $\Delta^2 u=0$ on $B_{2a}$. Then there exists a constant $C=C(n)>0$ such that, for every pair of radii $0<\dfrac{1}{8}r<s<r\le 2a$,
\[
\int_{B_s}|\Delta u|^2\,dV
\leq \frac{C}{(r-s)^2}\int_{B_r}|\nabla u|^2\,dV
\;+\; \frac12\int_{B_r} |\Delta u|^2\,dV.
\]
\end{lemma}
\begin{proof}
    Set $f:=\Delta u$. Since $u$ is biharmonic on $B_{2a}$, $f$ is harmonic on $B_{2a}$ and $|f|$ is subharmonic. 
By Bochner’s formula under $\mathrm{Ric}\ge 0$, we have
\begin{equation*}\label{eq:bochner-K-general}
\frac12\,\Delta|\nabla u|^2
=|\nabla^2u|^2+\langle\nabla u,\nabla f\rangle+\mathrm{Ric}(\nabla u,\nabla u)
\;\ge\;\frac1n f^2+\langle\nabla u,\nabla f\rangle,
\end{equation*}
hence
\begin{equation}\label{eq:f2-ineq-general}
f^2 \;\le\; \frac n2\,\Delta|\nabla u|^2 \;-\; n\,\langle\nabla u,\nabla f\rangle.
\end{equation}

Fix $0<s<r\le 2a$ and choose a cutoff function $\varphi\in C_c^\infty(B_r)$ satisfying
\[
\varphi\equiv1 \text{ on }B_s,\quad 0\le\varphi\le1,\quad
|\nabla\varphi|\le \frac{C_0}{r-s}\varphi,\quad
|\Delta\varphi|\le \frac{C_0}{(r-s)^2}.
\]

Multiplying \eqref{eq:f2-ineq-general} by $\varphi^2$ and integrating over $B_r$ gives
\begin{align}
\int_{B_r}\varphi^2 f^2\,dV
&\le \frac n2\int_{B_r}\varphi^2\,\Delta|\nabla u|^2\,dV
      -n\int_{B_r}\varphi^2\langle\nabla u,\nabla f\rangle\,dV.
\label{eq:master-all}
\end{align}

We work on the two terms of the righthand side of the previous inequality.
We start with the term
\[
\int_{B_r}\varphi^2\,\Delta|\nabla u|^2\,dV.
\]
Integrating by parts twice yields
\begin{eqnarray*}
\int_{B_r}\varphi^2\,\Delta|\nabla u|^2\,dV
&=&\int_{B_r}|\nabla u|^2\,\Delta(\varphi^2)\,dV\\
&=&\int_{B_r}|\nabla u|^2\big(2\varphi\,\Delta\varphi+2|\nabla\varphi|^2\big)\,dV \nonumber\\
&\le& \frac{C_1}{(r-s)^2}\int_{B_r}|\nabla u|^2\,dV,
\label{eq:lap-all}
\end{eqnarray*}
where $C_1 = 2C_0+2C_0^2$.

\medskip
We now proceed with the term
\[
n\int_{B_r}\varphi^2\langle\nabla u,\nabla f\rangle\,dV.
\]
For any $\eta>0$, by Young’s inequality we have
\begin{align*}
\Big|n\int_{B(r)}\varphi^2\langle\nabla u,\nabla f\rangle\,dV\Big|
&\le \frac{n\eta}{2}\int_{B(r)}\varphi^2|\nabla u|^2\,dV
\;+\; \frac{n}{2\eta}\int_{B(r)}\varphi^2|\nabla f|^2\,dV.
\label{eq:young-all}
\end{align*}
Since $f$ is harmonic in $B_r$, testing the equation $\Delta f=0$ with $\varphi^2 f$ gives the standard Caccioppoli inequality:
\begin{equation}\label{eq:caccio-f-corrected}
\int_{B_r}\varphi^2|\nabla f|^2\,dV
\leq \frac{C_2}{(r-s)^2}\int_{B_r} f^2\,dV.
\end{equation}
Substituting \eqref{eq:caccio-f-corrected} into the previous inequality gives
\begin{equation}\label{eq:mixed-all}
\Big|n\int_{B_r}\varphi^2\langle\nabla u,\nabla f\rangle\,dV\Big|
\;\le\; \frac{n\eta}{2}\int_{B_r}\varphi^2|\nabla u|^2\,dV
\;+\; \frac{nC_2}{2\eta(r-s)^2}\int_{B_r} f^2\,dV.
\end{equation}

\medskip
Combining \eqref{eq:master-all}, \eqref{eq:caccio-f-corrected} and \eqref{eq:mixed-all} yields
\[
\int_{B_r}\varphi^2 f^2\,dV
\;\le\; \left(\frac{C_1}{(r-s)^2}+\frac{n\eta}{2}\right)\int_{B_r}|\nabla u|^2\,dV
\;+\; \frac{nC_2}{2\eta(r-s)^2}\int_{B_r} f^2\,dV.
\]
Choose $\eta=\dfrac{nC_2}{(r-s)^2}$ so that the coefficient in front of $\int_{B_r} f^2$ equals $\tfrac12$. Then
\[
\int_{B_r}\varphi^2 f^2\,dV
\leq \frac{C_3}{(r-s)^2}\int_{B_r}|\nabla u|^2\,dV
\;+\; \frac12\int_{B_r} f^2\,dV,
\]
where $C_3=C_1 +\dfrac{n^2 C_2}{2}$.
Since $\varphi\equiv1$ on $B_s$, the left-hand side equals $\int_{B_s} f^2\,dV$. Then
\[
\int_{B_s}|\Delta u|^2\,dV
\leq \frac{C}{(r-s)^2}\int_{B_r}|\nabla u|^2\,dV
\;+\; \frac12\int_{B_r} |\Delta u|^2\,dV,
\]
which is what we wanted to prove.

\end{proof}

The next lemma is a Caccioppoli-type inequality, and it holds for harmonic functions.
It comes as a bit of a surprise that it also holds for biharmonic functions.
\begin{lemma}
\label{lem:Caccio-1-biharmonic-Br}
Let $(M,g)$ be a smooth Riemannian manifold with a pole $p\in M$, 
such that $\mbox{Ric}\geq 0$.
Let $0<\dfrac{1}{8}r<s<r<\infty$
and let $B_\rho:=B_\rho(p)$. If $u\in C^\infty(B_{2r})$ is \emph{biharmonic}, i.e.\ $\Delta^2 u=0$ in $B_{2r}$, then there exists a constant $C=C(n)>0$ such that
\begin{equation}\label{eq:Caccio-int-Br}
\int_{B_s} |\nabla u|^2\,dV \;\le\; \frac{C}{(r-s)^2}\int_{B_r} u^2\,dV.
\end{equation}
\end{lemma}
\begin{proof}
Let $\varphi\in C^{\infty}_{c}\left(B_r\right)$ be a cutoff function 
as in Corollary \ref{corollary:special_cutoff}.
    By integration by parts on $B_r$,
\begin{align*}
\int_{B_r} \varphi^2\,|\nabla u|^2\,dV
&= -\int_{B_r} \varphi^2\,u\,\Delta u\,dV
   - 2\int_{B_r} \varphi\,u\,\langle\nabla\varphi,\nabla u\rangle\,dV.
\label{eq:1st-energy-identity}
\end{align*}
Applying Young’s inequality to the last term, for any $\eta>0$,
\[
\left|2\int_{B_r} \varphi\,u\,\langle\nabla\varphi,\nabla u\rangle\,dV\right|
\le \eta \int_{B_r} \varphi^2 |\nabla u|^2\,dV
   + \frac{C}{\eta (r-s)^2}\int_{B_r} u^2\,dV,
\]
where we used $|\nabla\varphi|\le C/(r-s)$.
Choosing $\eta=\tfrac12$ yields,
\begin{equation}\label{eq:prelim}
\int_{B_r} \varphi^2|\nabla u|^2\,dV
\;\le\;
2\int_{B_r} \varphi^2|u\,\Delta u|\,dV
+\frac{C}{(r-s)^2}\int_{B_r} u^2\,dV.
\end{equation}

\medskip
Our purpose now is to estimate the first term on the righthand side of \eqref{eq:prelim}:
\[
2\int_{B_r} \varphi^2|u\,\Delta u|\,dV.
\]
To do so, requires some preliminary estimates. First of all,
by Young's inequality
\[
2\int_{B_r} \varphi^2|u\,\Delta u|\,dV
\le \int_{B_r}\varepsilon^2 \varphi^2|\Delta u|^2\,dV
  + \int_{B_r} \frac{1}{\varepsilon^2}\varphi^2 u^2\,dV.
\]
We must remark here that \emph{$\varepsilon$ is a function which will depend on the cutoff $\varphi$}, not a constant
as usual: this will
be reflected in the choice we shall make later. To simplify
the writing, we let
\[
\psi \colon = \varepsilon \varphi.
\]
We first deal with the term 
\[
\int_{B_r} \psi^2|\Delta u|^2\,dV.
\]
Since $\Delta^2 u=0$ in $B_r$, for the test function $\varphi^2 u$ we have
\[
0=\int_{B_r} \Delta^2 u\,(\psi^2 u)\,dV
=\int_{B_r} \Delta u\,\Delta(\psi^2 u)\,dV.
\]
Expanding
\[
\Delta(\psi^2 u)=\psi^2\Delta u+4\psi\langle\nabla\psi,\nabla u\rangle
+(\Delta(\psi^2))\,u,
\]
we obtain the identity
\begin{equation}\label{eq:D2-identity}
\int_{B_r} \psi^2\,|\Delta u|^2\,dV
= -4\int_{B_r} \psi\,\Delta u\,\langle\nabla\psi,\nabla u\rangle\,dV
  -\int_{B_r} \Delta u\,u\,\Delta(\psi^2)\,dV.
\end{equation}

\medskip
Let us estimate the first term on the righthand side of \eqref{eq:D2-identity}.
Using Young’s inequality with $a=\psi|\Delta u|$ and $b=2|\nabla\psi||\nabla u|$,
\begin{align*}
4\left|\int_{B_r} \psi\,\Delta u\,\langle\nabla\psi,\nabla u\rangle\,dV\right|
&\le \frac14\int_{B_r} \psi^2|\Delta u|^2\,dV
      +16\int_{B_r} |\nabla\psi|^2|\nabla u|^2\,dV
\nonumber\\
%&\le \frac14\int_{B_r} \varphi^2|\Delta u|^2\,dV
   %   +\frac{C}{(r-s)^2}\int_{B_r} \varphi^2|\nabla u|^2\,dV,
%\label{eq:mixed-term-est}
\end{align*}
%where in the last inequality we used the properties of
%$\varphi$ established in Corollary \ref{corollary:special_cutoff}.

\medskip
Next, we work with the second term of \eqref{eq:D2-identity}.
%First, we have that
%$$|\Delta(\varphi^2)|\le C\big(|\nabla\varphi|^2+|\varphi||\Delta\varphi|\big)\le C/(r-s)^2.$$
Using Young's inequality with a parameter $\eta>0$,
\begin{align*}
\left|\int_{B_r} \Delta u\,u\,\Delta(\psi^2)\,dV\right|
&= \left|\int_{B_r} (\psi\Delta u)\,\frac{\Delta(\psi^2)}{\psi}\,u\,dV\right|
\nonumber\\
&\le \frac{\eta}{2}\int_{B_r} \psi^2|\Delta u|^2\,dV
 %  + \frac{C}{\eta(r-s)^4}\int_{B_r} u^2\,dV,
 +\frac{1}{2\eta}\int_{B_r}\left(\frac{\Delta(\psi^2)}{\psi}\right)^2 u^2\,dV.
\label{eq:delta-phi2-est}
\end{align*}
%where we used the bound $\dfrac{|\Delta(\varphi^2)|^2}{\varphi^2}\le C(r-s)^{-4}%(see Corollary \ref{corollary:special_cutoff}).

\medskip
Inserting the two previous estimates into \eqref{eq:D2-identity}
%\begin{equation}\label{eq:D2-final}
%\int_{B_r} \varphi^2|\Delta u|^2\,dV
%\;\le\;
%\frac{C}{(r-s)^2}\int_{B_r} \varphi^2|\nabla u|^2\,dV
%\;+\;\frac{C}{(r-s)^4}\int_{B_r} u^2\,dV.
%\end{equation}
\begin{eqnarray*}%\label{eq:D2-final}
\int_{B_r} \psi^2|\Delta u|^2\,dV
&\le&
\left(\frac14+\frac{\eta}{2}\right)\int_{B_r} \psi^2|\Delta u|^2\,dV
      +16\int_{B_r} |\nabla\psi|^2|\nabla u|^2\,dV\\
      &&
      +\frac{1}{2\eta}\int_{B_r}\left(\frac{\Delta(\psi^2)}{\psi}\right)^2 u^2\,dV.
\end{eqnarray*}
Taking $\eta=\frac{1}{2}$ and absorbing the term $\int_{B_r} \psi^2|\Delta u|^2\,dV$
to the left yields
\begin{equation}\label{eq:D2-final}
\int_{B_r} \psi^2|\Delta u|^2\,dV
\;\le\;
      32\int_{B_r} |\nabla\psi|^2|\nabla u|^2\,dV
      +2\int_{B_r}\left(\frac{\Delta(\psi^2)}{\psi}\right)^2 u^2\,dV.
\end{equation}
\medskip
Therefore, using \eqref{eq:D2-final}, we get the following estimate for
the first term on the righthand side of \eqref{eq:prelim}:
\begin{eqnarray*}
2\int_{B_r} \varphi^2|u\,\Delta u|\,dV&\leq & 32\int_{B_r} |\nabla\psi|^2|\nabla u|^2\,dV
      +2\int_{B_r}\left(\frac{\Delta(\psi^2)}{\psi}\right)^2 u^2\,dV\\
&&+
\int_{B_r} \frac{1}{\varepsilon^2}\varphi^2 u^2\,dV.
\end{eqnarray*}
Using this estimate in \eqref{eq:prelim} and collecting terms yields
%\begin{eqnarray*}
%\int_{B_r} \varphi^2|\nabla u|^2\,dV
%&\le& \varepsilon\frac{C}{(r-s)^2}\int_{B_r} \varphi^2|\nabla u|^2\,dV\\
% &&  +\frac{C}{\varepsilon(r-s)^4}\int_{B_r} u^2\,dV
%   +\frac{C}{(r-s)^2}\int_{B_r} u^2\,dV.
%\end{eqnarray*}
\begin{eqnarray*}
\int_{B_r} \varphi^2|\nabla u|^2\,dV
&\le& 
32\int_{B_r} |\nabla\psi|^2|\nabla u|^2\,dV
      +2\int_{B_r}\left(\frac{\Delta(\psi^2)}{\psi}\right)^2 u^2\,dV\\
 &&  +
\int_{B_r} \frac{1}{\varepsilon^2}\varphi^2 u^2\,dV
+\frac{C}{(r-s)^2}\int_{B_r} u^2\,dV.
\end{eqnarray*}

Choose $\varepsilon=\dfrac{\varphi(r-s)}{2L}$ (where $L>0$ is to be chosen later).
Then we can estimate
\begin{eqnarray*}
    \int_{B_r} |\nabla\psi|^2|\nabla u|^2\,dV&=&\frac{\left(r-s\right)^2}{L^2}\int_{B_r}
    \varphi^2\left|\nabla \varphi\right|^2|\nabla u|^2\,dV\\
    &\leq&
    \frac{\left(r-s\right)^2}{L^2}\frac{C}{\left(r-s\right)^2}\int_{B_r}\varphi^3 |\nabla u|^2\,dV \quad \mbox{(by Corollary \ref{corollary:special_cutoff})}\\
    &\leq&
    \frac{C}{L^2}\int_{B_r}\varphi^2 |\nabla u|^2\,dV,
\end{eqnarray*}
where in the last inequality we used the properties of
$\varphi$ established in Corollary \ref{corollary:special_cutoff}. Hence
\[
 \int_{B_r} |\nabla\psi|^2|\nabla u|^2\,dV \leq \frac{C}{L^2}\int_{B_r}\varphi^2 |\nabla u|^2\,dV.
\]
Thus, if we choose $L$ such that $C/L^2=\frac{1}{64}$, we obtain
\begin{eqnarray*}
\int_{B_r} \varphi^2|\nabla u|^2\,dV
&\le& 
\frac{1}{2}\int_{B_r} \varphi^2|\nabla u|^2\,dV
      +2\int_{B_r}\left(\frac{\Delta(\psi^2)}{\psi}\right)^2 u^2\,dV\\
 &&  +
\frac{C}{16\left(r-s\right)^2}\int_{B_r}  u^2\,dV
+\frac{C}{(r-s)^2}\int_{B_r} u^2\,dV.
\end{eqnarray*}
Next, we estimate the term $\int_{B_r}\left(\frac{\Delta(\psi^2)}{\psi}\right)^2 u^2\,dV$.
\begin{eqnarray*}
    \int_{B_r}\left(\frac{\Delta(\psi^2)}{\psi}\right)^2 u^2\,dV&=& \frac{\left(r-s\right)^2}{4L^2}
    \int_{B_r} \left(\frac{\Delta(\varphi^4)}{\varphi^2}\right)^2 u^2\,dV\\
    &\leq& \frac{\left(r-s\right)^2}{4L^2} \int_{B_r}\frac{C}{\left(r-s\right)^4} u^2\,dV
    \quad \mbox{(by Corollary \ref{corollary:special_cutoff})}\\
    &\leq& \frac{1}{256\left(r-s\right)^2}\int_{B_r} u^2\,dV.
\end{eqnarray*}
Therefore
\[
\int_{B_r} \varphi^2|\nabla u|^2\,dV\leq 
\frac{1}{2}\int_{B_r} \varphi^2|\nabla u|^2\,dV
+
 \frac{C'}{\left(r-s\right)^2}\int_{B_r} u^2\,dV.
\]
%so that
%$\varepsilon\,\dfrac{C}{(r-s)^2}=\tfrac12$. 
Moving the term
$\tfrac12\int \varphi^2|\nabla u|^2$ to the left yields
\[
\frac12\int_{B_r} \varphi^2|\nabla u|^2\,dV
\;\le\; \frac{C''}{(r-s)^2}\int_{B_r} u^2\,dV,
\]
and therefore
\begin{equation*}\label{eq:phi-grad-final}
\int_{B_r} \varphi^2|\nabla u|^2\,dV
\;\le\; \frac{C'''}{(r-s)^2}\int_{B_r} u^2\,dV.
\end{equation*}

\medskip
Since $\varphi\equiv 1$ on $B_s$, this implies
\[
\int_{B_s} |\nabla u|^2\,dV
\;\le\; \frac{C}{(r-s)^2}\int_{B_r} u^2\,dV,
\]
which is \eqref{eq:Caccio-int-Br}.
\end{proof}

\subsection{Proof of Theorem \ref{thm:main}}
\label{sect:proof_main_thm}
With all the previous preparation the proof of our main result is not difficult. 
Indeed, by Lemma \ref{lem:laplacian_estimate} and the Hole-filling Lemma
(Lemma \ref{lemma:hole-filling}), we have
that for every $0<\dfrac{1}{8}r<s<r<\infty$ 
\[
\int_{B_s}|\Delta u|^2\,dV
\leq \frac{C}{(r-s)^2}\int_{B_r}|\nabla u|^2\,dV.
\]
Thus, 
taking $s=\frac{1}{2}r$
\[
\int_{B_{\frac{1}{2}r}}|\Delta u|^2\,dV
\leq \frac{C'}{r^2}\int_{B_r}|\nabla u|^2\,dV,
\]
and then, from Lemma \ref{lem:Caccio-1-biharmonic-Br} we have the following Caccioppoli-type inequality 
for biharmonic functions:
\begin{equation}
\label{ineq:Cacciopoli_laplacian}
\int_{B_{\frac{1}{2}r}}|\Delta u|^2\,dV
\leq \frac{C''}{r^4}\int_{B_{2r}}u^2\,dV.
\end{equation}
By the volume doubling property (which holds if $\mbox{Ric}\geq 0$), we have
\begin{eqnarray*}
\frac{1}{\mbox{Vol}\left(B_{\frac{1}{2}r}\right)}\int_{B_{\frac{1}{2}r}}|\Delta u|^2\,dV
&\leq& \frac{C'''}{r^4}\frac{1}{\mbox{Vol}\left(B_{2r}\right)}\int_{B_{2r}}u^2\,dV,
\end{eqnarray*}
and hence, since $\Delta u$ is harmonic, by the mean value inequality (Lemma \ref{lemma:mean_value_ineq}),
\begin{eqnarray*}
\sup_{p\in B_{\frac{1}{4}r}}\left|\Delta u\right|^2 &\leq& 
\frac{C''''}{r^4}\frac{1}{\mbox{Vol}\left(B_{2r}\right)}\int_{B_{2r}}u^2\,dV\\
&\leq& \frac{C^{(5)}}{r^4}\sup_{x\in B_r\left(p\right)}\left|u\left(x\right)\right|^2
=\frac{1}{r^4}o\left(r^4\right),
\end{eqnarray*}
and hence if $u$ is of subquadratic growth, letting $r\rightarrow \infty$, it 
follows that
\[
\Delta u \equiv 0,
\]
i.e., $u$ is harmonic, which is what we wanted to prove.

\end{document}